\documentclass[12pt]{amsart}

\usepackage{amssymb,amsfonts,amsmath,amsopn,amstext,amscd,latexsym,fullpage,verbatim}
\usepackage{amsthm}

\theoremstyle{remark}
\newtheorem{para}{\bf}[section]

\newtheorem{rmk}[para]{\bf Remark}

\theoremstyle{definition}
\newtheorem{exam}[para]{\bf Example}

\newtheorem{dfn}[para]{\bf Definition}

\theoremstyle{plain}
\newtheorem{thm}[para]{\bf Theorem}
\newtheorem{lemma}[para]{\bf Lemma}

\newtheorem{prop}[para]{\bf Proposition}

\newenvironment{numequation}
{\addtocounter{enumi}{1}\begin{equation}}{\end{equation}}

\newcommand{\cA}{{\mathcal A}}

\newcommand{\cD}{{\mathcal D}}
\newcommand{\cE}{{\mathcal E}}
\newcommand{\cF}{{\mathcal F}}

\newcommand{\cM}{{\mathcal M}}

\newcommand{\cO}{{\mathcal O}}

\newcommand{\cU}{{\mathcal U}}

\newcommand{\cX}{{\mathcal X}}
\newcommand{\cY}{{\mathcal Y}}

\newcommand{\bbC}{{\mathbb C}}

\newcommand{\bbF}{{\mathbb F}}
\newcommand{\bbG}{{\mathbb G}}

\newcommand{\bbN}{{\mathbb N}}

\newcommand{\bbP}{{\mathbb P}}

\newcommand{\bbZ}{{\mathbb Z}}

\newcommand{\frb}{{\mathfrak b}}

\newcommand{\frd}{{\mathfrak d}}

\newcommand{\frg}{{\mathfrak g}}
\newcommand{\frh}{{\mathfrak h}}

\newcommand{\frl}{{\mathfrak l}}

\newcommand{\frp}{{\mathfrak p}}
\newcommand{\frq}{{\mathfrak q}}

\newcommand{\frt}{{\mathfrak t}}
\newcommand{\fru}{{\mathfrak u}}

\newcommand{\frz}{{\mathfrak z}}

\newcommand{\GL}{{\rm GL}}
\newcommand{\Hom}{{\rm Hom}}

\newcommand{\Ad}{{\rm Ad}}

\newcommand{\hra}{\hookrightarrow}

\newcommand{\ind}{{\rm ind}}
\newcommand{\Ind}{{\rm Ind}}

\newcommand{\lra}{\longrightarrow}
\newcommand{\midc}{\;|\;}

\newcommand{\ra}{\rightarrow}
\newcommand{\Rep}{{\rm Rep}}

\newcommand{\sub}{\subset}

\newcommand{\alg}{{\rm alg}}

\newcommand{\Pf}{{\it Proof. }}
\renewcommand{\qed}{{\hfill{\space} $\Box$}}

\numberwithin{equation}{section}

\newcommand{\N}{\mathbb N}
\newcommand{\Z}{\mathbb Z}

\renewcommand{\P}{\mathbb P}

\newcommand{\dUg}{\dot{\cU}(\frg)}
\newcommand{\dUb}{\dot{\cU}(\frb)}
\newcommand{\dUp}{\dot{\cU}(\frp)}
\newcommand{\dUq}{\dot{\cU}(\frq)}
\newcommand{\dUu}{\dot{\cU}(\fru)}
\newcommand{\dUt}{\dot{\cU}(\frt)}
\newcommand{\dUup}{\dot{\cU}(\fru_P)}
\newcommand{\dUump}{\dot{\cU}(\fru^-_P)}
\newcommand{\dUl}{\dot{\cU}(\frl_P)}
\newcommand{\dUum}{\dot{\cU}(\fru^-)}
\newcommand{\dcO}{\dot{\cO}}

\begin{document}

\title{Equivariant vector bundles on Drinfeld's halfspace over a finite field}

\author{Sascha Orlik}
\address{Fachgruppe Mathematik und Informatik, Bergische  Universit\"at Wuppertal, Gau\ss{}stra\ss{}e
D-42119 Wuppertal,  Germany}
\email{orlik@math.uni-wuppertal.de}
\date{}
\maketitle

\begin{abstract}
Let $\cX \subset \P_k^d$ be Drinfeld's halfspace over a finite field $k$ and let $\cE$ be a  homogeneous vector bundle on $\P_k^d$.  
The paper deals with two different descriptions of the space of global sections $H^0(\cX,\cE)$ as $\GL_{d+1}(k)$-representation. This is an infinite dimensional modular $G$-representation. Here we follow the ideas of \cite{O2,OS} treating the $p$-adic case. As a replacement for the universal enveloping algebra we consider both the crystalline universal enveloping algebra and the ring of differential operators on the flag variety with respect to $\cE.$

\end{abstract}

\normalsize

\section*{Introduction}
Let $k$ be a finite field and denote by $\cX$ Drinfeld's halfspace of dimension $d\geq1$ over $k$. This is the complement of all  $k$-rational hyperplanes in projective space $\mathbb \P_k^d,$ i.e.,
$$ \cX=\mathbb \P_k^d\setminus \bigcup\nolimits_{H \varsubsetneq k^{d+1}}\mathbb \P(H).$$ 
This object  is equipped with an action of $G={\rm GL}_{d+1}(k)$ and can be viewed as a Deligne-Lusztig variety, as well as a period domain over a finite field \cite{OR}.
In particular we get for every homogeneous vector bundle $\cE$ on $\P^d_k$ an induced action of $G$ on the space of global sections $H^0(\cX,\cE)$ which is an infinite-dimensional modular $G$-representation. 

In \cite{O2} we considered the same problem for the Drinfeld halfspace over a $p$-adic field $K$. We constructed for every homogeneous vector bundle $\cE$  a filtration
by closed ${\rm GL}_{d+1}(K)$-subspaces and determined the graded pieces in terms of locally analytic $G$-representations in the sense of Schneider and Teitelbaum \cite{ST1}. The definition of the filtration above involves the geometry of $\cX$ being the complement of an hyperplane arrangement.
In the $p$-adic case $H^0(\cX, \cE)$ is a "bigger" object, it is a reflexive $K$-Fr\'echet space with a continuous $G$-action. Its strong dual  is a locally analytic $G$-representation. The interest here for studying those objects lies in the connection to the $p$-adic Langland correspondence.

In his thesis \cite{Ku} Kuschkowitz adapts the strategy of the $p$-adic case to the situation considered here.

{\it \noindent {\bf Theorem} (Kuschkowitz):} Let $\cE$ be a homogeneous vector bundle on $\P^d_k.$  There is a filtration $$\cE(\cX)^0 \supset \cE(\cX)^{1} \supset \cdots \supset \cE(\cX)^{d-1} \supset \cE(\cX)^d =  H^0(\P^d,\cE)$$
on $\cE(\cX)^0=H^0(\cX,\cE)$ such that for $j=0,\ldots,d-1,$ there is an
extension of $G$-representa\-tions
$$0 \rightarrow   \Ind^G_{P_{(j+1,d-j)}}(\tilde{H}^{d-j}_{\bbP^j}(\bbP^d,\cE)\otimes St_{d+1-j}) \rightarrow \cE(\cX)^{j}/\cE(X)^{j+1} \rightarrow v^{G}_{P_{(j+1,1,\ldots,1)}} \otimes H^{d-j}(\P^d_k,\cE) \rightarrow 0. $$ 

\noindent Here the module $v^{G}_{{P_{(j+1,1,\ldots,1)}}}$ is a generalized Steinberg representation corresponding to the decomposition $(j+1,1,\ldots,1)$ of $d+1.$
Further $P_{\underline{j}}=P_{(j,d+1-j)}\subset G$ is the (lower) standard-parabolic  subgroup attached to the decomposition $(j,d+1-j)$ of $d+1$ and ${\rm St}_{d+1-j}$ is the Steinberg representation of ${\rm GL}_{d+1-j}(k).$
Here the action of the parabolic is induced by the composite 
$$ P_{(j,d+1-j)} \to L_{(j,d+1-j)}= {\rm GL}_j(k) \times {\rm GL}_{d+1-j}(k) \to {\rm GL}_{d+1-j}(k).$$ 
Finally we have the reduced local cohomology
$$\tilde{H}^{d-j}_{\P^j_k} (\P^d_k ,\cE):={\rm ker}\,\Big(H^{d-j}_{\P^j_k}(\P^d_k,\cE) \rightarrow H^{d-j}(\P^d_k,\cE)\Big)$$
which is a $P_{(j+1,d-j)}$-module.

In  the $p$-adic setting the substitute of the LHS of this extension has the structure of an admissible module over the locally analytic distribution algebra. Here we were able to give a description of the dual representation in terms of a series of functors
$$\cF^G_P:\cO^\frp_\alg \times \Rep^{\infty,{\it adm}}_K(P) \to \Rep^{\ell a}_K(G)$$
where $\cO^\frp_\alg$ consist of the algebraic objects of type $\frp$ in the category $\cO$, $\Rep^{\infty,{\it adm}}_K(P)$
is the category of smooth admissible $P$-representations and 
$\Rep^{\ell a}_K(G)$ denotes the category of locally analytic $G$-representations.

In positive characteristic Lie algebra methods do not behave so nice. E.g. the local cohomology groups are  not finitely generated over the universal enveloping algebra of the Lie algebra of $\GL_{d+1}$ so that the same  machinery does not apply. Our goal in this paper is to concentrate on the latter aspect and to present two candidates for a substitution in this situation. The first approach considers the crystalline universal enveloping algebra $\dUg$ (or Kostant form) which coincides with the distribution algebra of $G$, cf. \cite{Ja}. The action of $\frg$ extends  to one of $\dUg$, so that 
$H^0(\cX,\cE)$ becomes a module over the smash product $k[G] \# \dUg.$ We define a positive characteristic version of $\cF^G_P$ and prove analogously properties of them as in the $p$-adic case, e.g. we give an irreducibility criterion, cf. \cite{OS}. 

The second approach uses instead of $\dUg$ the ring of distributions $D^\cE$ on the flag variety with respect to $\cE.$
The important point is that the natural map $\dUg \to D^\cE$ is in contrast to the field of complex numbers not surjective as shown  by Smith \cite{Sm}.
We will show that the above local cohomology modules are finitely generated leading to a category $\cO_{D^\cE}$ where we can define similar our functors $\cF^G_P$.

\vspace{0.7cm}
{\it Notation:} We let  $p$ be a prime number, $q=p^n$ some power and let $k=\bbF_q$ the corresponding field with $q$ elements. We fix an algebraic closure $\bbF:=\overline{\bbF}_q$ and denote by $\P^d_\bbF$ the projective space of dimension $d$ over $\bbF.$ If $Y \subset \P^d_\bbF$ is a closed algebraic $\bbF$-subvariety and $\cF$ is a sheaf on $\P^d_\bbF$ we write $H_Y^\ast(\P^d_\bbF,\cF)$ for the corresponding local cohomology.  We  consider the algebraic action ${\bf G} \times \P^d_\bbF \rightarrow \P^d_\bbF$ of ${\bf G}$ on $\P_\bbF^d$ given by
$$g\cdot [q_0:\cdots :q_d]:=m(g,[q_0:\cdots :q_d]):= [q_0:\cdots :q_d]g^{-1}.$$
 
We use bold letters $\bf H$ to denote algebraic group schemes over $\bbF_q$, whereas we use normal letters $H$ for their $\bbF_q$-valued points. We denote by ${\bf H}_\bbF:= {\bf H} \times_{\bbF_q} \bbF$  their base change to $\bbF.$
We use Gothic letters $\frh$ for their Lie algebras (over $\bbF$). The corresponding enveloping algebras are denoted as usual by $U(\frh).$

We denote by ${\bf G}_{\bbZ}$ a split reductive algebraic group over $\bbZ$. We fix a Borel subgroup ${\bf B_{\bbZ} \subset G_{\bbZ}}$ and let  ${\bf U_{\bbZ}}$ be its unipotent radical and ${\bf U^-_{\bbZ}}$ the opposite radical. Let ${\bf T_{\bbZ}\subset G_{\bbZ}}$ be a fixed split torus 
and denote the root system by $\Phi$ and its subset of simple roots by 
$\Delta$.

\vspace{0.7cm}
{\it Acknowledgments:} 
I am grateful to Georg Linden for pointing out to me the paper of Smith \cite{Sm}.
This research was conducted in the framework of the research training group
\emph{GRK 2240: Algebro-Geometric Methods in Algebra, Arithmetic and Topology}, which is funded by the DFG.

\section{The theorem of Kuschkowitz}
In this section we recall shortly the strategy for proving the theorem  of Kuschkowitz. Here we consider for ${\bf G}$ the general linear group ${\rm {\bf GL_{d+1}}}$ and for 
${\bf B\subset G}$ the Borel subgroup of lower triangular matrices and ${\bf T}$ the diagonal torus. Denote by  ${\bf
\overline{T}}$ its image in ${\rm \bf PGL}_{d+1}.$ For $0 \leq i \leq
d,$ let $\epsilon_i:{\bf T}\rightarrow {\bf \bbG_m}$ be the
character defined by $\epsilon_i({\rm diag}(t_1,\ldots,t_d))= t_i.$
Put $\alpha_{i,j}:=\epsilon_i - \epsilon_j$ for $i\neq j$. Then $\Delta:=\{\alpha_{i,i+1} \,\mid\, 0\leq i \leq d-1\}$ are the simple roots
and $\Phi:=\{\alpha_{i,j}\,\mid\, 0\leq i\neq j \leq d-1\}$ are
the roots of ${\bf G}$ with respect to ${\bf T\subset B}$. For a
decomposition $(i_1,\ldots,i_r)$ of $d+1,$ let ${\bf
P_{(i_1,\ldots,i_r)}}$ be the corresponding standard-parabolic
subgroup of ${\bf G}$, ${\bf U_{(i_1,\ldots,i_r)}}$ its unipotent radical and ${\bf
L_{(i_1,\ldots,i_r)}}$ its Levi component.

\bigskip
Let $\cE$ be a homogeneous  vector bundle  on $\P_\bbF^d$. Our finite  group $G$ stabilizes $\cX$. Therefore, we obtain an
induced action of $G$ on the $\bbF$-vector space of global sections $\cE(\cX).$ Further $\cE$ is  naturally a $\frg$-module, i.e., there is a homomorphism of Lie algebras
$\frg \rightarrow {\rm End}(\cE)$.
For the structure sheaf $\cO=\cO_{\P^d_\bbF}$ with its natural ${\bf G}$-linearization we can describe the action of $\frg$ on $\cO(\cX).$ Indeed, for a root $\alpha=\alpha_{i,j}\in \Phi,$ let
$$L_\alpha :=L_{(i,j)} \in \frg_\alpha$$ be the standard generator of the
weight space $\frg_\alpha$  in $\frg.$ Let $\mu \in
X^\ast(\overline{\bf T})$ be a character of the torus $\overline{\bf
T}.$ Write $\mu$ in the shape $\mu=\sum_{i=0}^d m_i \epsilon_i$ with
$\sum_{i=0}^d m_i=0.$ Define $\Xi_{\mu} \in \cO(\cX)$  by
$$\Xi_\mu(q_0,\ldots,q_d)=q_0^{m_0}\cdots q_d^{m_d}.$$
For these functions, the action  of $\frg$  is given by
\begin{equation}\label{Strukturgarbe}
L_{(i,j)}\cdot \Xi_\mu = m_j\cdot \Xi_{\mu + \alpha_{i,j}} 
\end{equation}
and
\begin{equation*}
t\cdot \Xi_\mu = (\sum\nolimits_ i m_i t_i)\cdot\Xi_\mu, \; t\in \frt.
\end{equation*}

Fix an integer $0 \leq j \leq d-1.$
Let $$\P^j_\bbF=V(X_{j+1},\ldots,X_{d})\subset \P^d_\bbF$$ be the closed subvariety
defined by the vanishing of the coordinates $X_{j+1},\ldots,X_{d}.$ The algebraic local cohomology modules $H^i_{\P^j_\bbF}(\P^d_\bbF,\cE), \; i\in \N$,
sit in a long exact sequence
$$\cdots \rightarrow H^{i-1}(\P^d_\bbF\setminus \P^j_\bbF,\cF) \rightarrow H^i_{\P^j_\bbF}(\P^d_\bbF,\cF) \rightarrow H^i(\P^d_\bbF,\cF) \rightarrow H^i(\P^d_\bbF\setminus \P^j_\bbF,\cF) \rightarrow \cdots$$
which is equivariant for the induced action of ${\bf P_{(j+1,d-j)}} \ltimes U(\frg)$. Here the semi-direct product is defined via the adjoint action of ${\bf P_{(j+1,d-j)}}$ on $\frg.$    We set
$$\tilde{H}^{d-j}_{\P^j_\bbF} (\P^d_\bbF ,\cE):={\rm ker}\,\Big(H^{d-j}_{\P^j_\bbF}(\P^d_\bbF,\cE) \rightarrow H^{d-j}(\P^d_\bbF,\cE)\Big)$$
which is consequently a ${\bf P_{(j+1,d-j)}} \ltimes U(\frg)$-module.

Consider the   exact sequence
of $\bbF$-vector spaces with $G$-action
$$0\rightarrow H^0(\P^d_\bbF,\cE) \rightarrow H^0(\cX,\cE)\rightarrow H^1_{\cY}(\P^d_\bbF,\cE) \rightarrow H^1(\P^d_\bbF,\cE)\rightarrow 0.$$
Note that the higher cohomology groups $H^i(\cX,\cE), \; i>0,$ vanish since $\cX$ is an affine space.
The $G$-representations $H^0(\P^d_\bbF,\cF),\;H^1(\P^d_\bbF,\cF)$ are finite-dimensional algebraic. Let $i: \cY \hookrightarrow
(\P^d_\bbF)$ denote the closed embedding and let $\bbZ$ be constant sheaf on $\cY.$  Then by \cite[Proposition 2.3
bis.]{SGA2}, we conclude that
$${\rm Ext}^\ast(i_\ast(\Z_{{\cY}}),\cE) =
H^\ast_{{\cY}}(\P^d_\bbF,\cE).$$
The idea is now to plug in a resolution of the sheaf $\bbZ$
on the boundary and works as follows.

Let $\{e_0,\ldots,e_d\}$ be the standard basis of $V=\bbF^{d+1}.$
For any $\alpha_i\in \Delta,$ put
$$V_i=\bigoplus^i_{j=0} \bbF\cdot e_j \; \mbox{ and } \; Y_i=\P(V_i)$$
For any subset $I\subset \Delta$ with  $\Delta\setminus I=\{\alpha_{i_1} < \ldots < \alpha_{i_r}\},$ set $Y_I=\P(V_{i_1})$ and consider it as a closed subvariety of $\P^d_\bbF$.
Furthermore, let $P_I$ be the lower parabolic subgroup of $G,$ such that  $I$
coincides with the simple roots appearing in the Levi factor of $P_I$. Hence the group $P_I$ stabilizes $Y_I.$
We obtain
\begin{eqnarray}
\cY=\bigcup_{g\in G}g\cdot Y_{ \Delta \setminus \{\alpha_{d-1}\}}.
\end{eqnarray}
Consider the natural closed embeddings 
$$\Phi_{g,I} :gY_I \longrightarrow {\cY}$$
and put
$$ \Z_{g,I}:=(\Phi_{g,I})_\ast(\Phi_{g,I}^\ast\, \Z).$$

We obtain the following  complex of sheaves on ${\cY}:$
\begin{eqnarray}
\nonumber 0 \rightarrow \Z \rightarrow\!\!\! \bigoplus_{I \subset
\Delta \atop |\Delta\setminus I|=1} \bigoplus_{g\in G/P_I} \Z_{g,I}
\rightarrow \!\!\!\bigoplus_{I \subset \Delta \atop |\Delta\setminus I|=2}
\bigoplus_{g\in G/P_I} \Z_{g,I} \rightarrow \cdots \rightarrow \!\!\!\bigoplus_{I \subset \Delta \atop |\Delta\setminus I|=i}
\bigoplus_{g\in G/P_I} \Z_{g,I} \rightarrow \cdots \\ \label{complex} \\ \nonumber \cdots\rightarrow
\!\!\!\bigoplus_{I \subset \Delta \atop |\Delta\setminus I|=d-1} \bigoplus_{g\in G/P_I} \Z_{g,I}
\rightarrow \bigoplus_{g\in G/P_\emptyset} \Z_{g,\emptyset} \rightarrow 0.
\end{eqnarray}
which is  acyclic by \cite{O1}.

In a next step one considers the spectral sequence which is induced by this complex  applied to ${\rm
Ext}^\ast(i_\ast(-),\; \cE).$ Here one uses that for all $I\subset \Delta,$ there is an isomorphism
$${\rm Ext}^\ast(i_\ast(\bigoplus\limits_{g\in G/P_I} \Z_{g,I}), \cE)
= \bigoplus_{g\in G/P_I} H_{gY_I}^{\ast}(\P^d_\bbF, \cE).$$

By evaluating the spectral sequence Kuschkowitz arrives in \cite{Ku} at the theorem mentioned in the introduction.

\section{First approach}
In this section we replace $U(\frg)$ by its crystalline version and transform the results of \cite{OS} to this setting. 

Let ${\bf G_{\bbZ}}$ be a split reductive algebraic group over $\bbZ$ and let  $\frg_\bbC$ be the Lie algebra of ${\bf G_{\bbZ}}(\bbC).$
On the other hand let $D(\textbf{G}_{\bbF})$ be the distribution algebra of  $\textbf{G}_{\bbF}=\textbf{G}_{\bbZ} \times_{\bbZ} \bbF.$ We identify $D(\textbf{G}_\bbF)$ with the  universal crystalline enveloping algebra (Kostant form) $\dUg.$ 
Thus $\dUg=\dUg_\bbZ \otimes \bbF$ where $\dUg_\bbZ$ is the $\bbZ$-subalgebra of $U(\frg_\bbC)$ generated by the expressions $$ x^{[n]}_\alpha:=x_\alpha^n/n! \mbox{, } y_\alpha^{[n]}:=y_\alpha^n/n! ,\; \alpha \in \Phi^+, n\in \bbN$$ 
$$\mbox{ and } {h_\alpha \choose n} ,\;\alpha \in \Delta, n\in \bbN ,$$
where $x_\alpha \in \frg_\alpha, y_\alpha \in \frg_{-\alpha}$ are generators and $h_\alpha=[x_\alpha,y_\alpha]$ for all $\alpha \in \Delta$.
We have a PBW-decomposition 
$$\dUg= \dUu \otimes_{\bbF} \dUt \otimes_{\bbF} \dUum$$
where the crystalline enveloping algebras for $\fru,\fru^-$ and $\frt$ are defined analogously. 

 We mimic the definition of the category $\cO$ in the sense of BGG.

 \begin{dfn}
  Let  $\dot{\cO}$ be the full subcategory of all  $\dUg$-modules such that 

i) $M$ is finitely generated as $\dUg$-module 

ii) $\dUt$  acts semisimple with finite-dimensional weight spaces.

iii) $\dUu$ acts locally finite-dimensional, i.e., for all $m\in M$ we have $\dim \dUu\cdot m < \infty.$
 \end{dfn}

 \begin{rmk}
  In \cite[Def. 3.2]{Hab} Haboush calls $\dUg$-modules  satisfying properties i) and ii) admissible. The category $\dot{\cO}$ has been also recently considered by Andersen \cite{An} and Fiebig \cite{Fi} (even  more generally for weight modules) discussing among others the structure of these objects.
 \end{rmk}

 Similarly, for a parabolic subgroup ${\bf P}\subset {\bf G}$ with Levi decomposition ${\bf P=L_P\cdot U_P}$ (induced by one over $\bbZ$), we let $\dcO^\frp$ be the full subcategory of $\dcO$ consisting of objects which are direct sums of finite-dimensional $\dUl$-modules. We let $\dcO_\alg$ be the full subcategory of $\dcO$ such that the action of $\dUt$ is induced on the weight spaces by  algebraic characters $X^\ast(T_\bbF) $ of $T_\bbF.$ Finally we set $$\dcO^\frp_\alg:= \dcO_\alg \cap \dcO^\frp.$$ As in the classical case there is for every object $M \in \dcO^\frp_\alg$  some finite-dimensional algebraic $P$-representation\footnote{Meaning that we restrict an algebraic ${\bf P}$-representation to the its rational points $P$.} $W\subset M$ which generates $M$ as a $\dUg$-module, i.e., there is a surjective homomorphism $\dUg \otimes_{\dUp} W \to M$. Again there is 
a PBW-decomposition $\dUg= \dUup \otimes_{\bbF} \dUl \otimes_{\bbF} \dUump$ such that the latter homomorphism can be seen as a map  $\dUump \otimes_{\bbF} W \to M.$

According to \cite{Hab} there is the notion of maximal vectors, highest weights, highest weight module etc.
and we may define Verma modules, cf. Def. 3.1 in loc.cit.  In fact let $\lambda$ be a one-dimensional $\dUt$-module. Then we consider it as usual via the trivial $\dUu$-action as a one-dimensional $\dUb$-module $\bbF_\lambda$. Then
$$M(\lambda)=\dUg \otimes_{\dUb} \bbF_\lambda$$
is the attached Verma module of weight $\lambda.$
As in the classical case  Theorem of \cite[1.2]{Hu} holds true for our highest weight modules. In particular it has a unique maximal proper submodule and therefore a unique simple quotient $L(\lambda)$, cf. \cite[Prop. 4.4]{Hab}, \cite[Thm 2.3]{An}, \cite[Prop. 2.3]{Fi}.

\begin{prop}
The simple modules in $\dcO_\alg$ are exactly of the shape $L(\lambda)$ for $\lambda \in X^\ast({\bf T}_\bbF).$
\end{prop}

\begin{proof}
 We need to show that every simple object in $\dcO_\alg$ is of this form. But by  \cite[Thm 4.9 i)]{Hab} simple admissible highest weight modules are of the form $L(\lambda)$ for a 
 one-dimensional $\dUt$-module $\lambda$. The algebraic condition forces $\lambda$ to be an algebraic character $\lambda \in X^\ast({\bf T}_\bbF)$.
\end{proof}

We also consider the full subcategory $M^d_{\dUg}$ for all $\dUg$-modules which satisfy condition ii) in the definition of $\dcO.$ For any such object $M$ we define a dual object $M'$ (the graded dual) following the classical concept: consider the weight space decomposition  $M=\bigoplus_\lambda M_\lambda$ where $\lambda$ is as above a one-dimensional $\dUt$-module.  Then the underlying vector space of $M'$ is the direct sum $\bigoplus_\lambda \Hom(M_\lambda,K).$ The $\dUg$-structure on it  is given by the natural one\footnote{Without the composition with the Cartan involution.}. Clearly one has $(M')'=M.$ 

We consider the natural action of $\fru_P^-$ on $\cO({\bf U_{P^-,\bbF}})$. This extends to a  non-degenerate pairing 
\begin{eqnarray}\label{pairin}
 \dUump \otimes \cO({\bf U^-_{P,\bbF}}) & \to & \bbF
\end{eqnarray}
such that $\cO({\bf U_{P^-,\bbF}})$ identifies with the graded dual of $\dUump.$ Moreover we pull back via this identification the action of $P$ on $(\dUg \otimes_{\dUp} 1)'$ to $\cO({\bf U_{P^-,\bbF}})$.  By construction we obtain the following statement.

\begin{lemma}
There is an isomorphism of $P \ltimes \dUg$-modules
$\cO({\bf U^-_{P,\bbF}}) \cong (\dUg \otimes_{\dUp} 1)'.$ \qed
\end{lemma}

The pairing (\ref{pairin}) extends for every algebraic $P$-representation $W$ to a pairing
\begin{eqnarray}\label{pairing}
 (\dUump \otimes W') \otimes (\cO({\bf U^-_{P,\bbF}}) \otimes W) & \to & \bbF
\end{eqnarray}
such that $\cO({\bf U^-_{P,\bbF}})\otimes W$ becomes isomorphic to $\dUump' \otimes W'$ as $P \ltimes \dUg$-modules.

Let $\dot{\bbF}[G,\frg]:=\bbF[G] \# \dUg$ be the smash product of $\dUg$ and the group algebra $\bbF[G]$ of $G.$ Recall that this $\bbF$-algebra has as underlying vector space the tensor product $\bbF[G] \otimes \dUg$ and where the multiplication is induced by $(g_1\otimes z_1)\cdot (g_2 \otimes z_2)=g_1g_2 \otimes Ad(g_2)(z_1)z_2$ for elements $g_i\in G, z_i \in \dUg, i=1,2.$

\vspace{0.5cm}

\begin{dfn}
 i) We denote by $Mod_{\dot{\bbF}[G,\frg]}^{d}$ be the full subcategory of all $\dot{\bbF}[G,\frg]$-modules for which the action of $\dUt$ is diagonalisable into finite-dimensional weight spaces.

 ii) We denote by $Mod_{\dot{\bbF}[G,\frg]}^{fg,d}$ be the full subcategory of $Mod_{\dot{\bbF}[G,\frg]}^{d}$ which are  finitely generated. 
\end{dfn}

For an object $\cM$ of $Mod_{\dot{\bbF}[G,\frg]}^d$ we define the dual  $\cM'$ as the graded dual of the underlying $\dUg$-module together with the contragradient action of $G.$

Let $M$ be an object of $\dcO^\frp_\alg$. Then there is a surjection 
$$p: \dUump \otimes W \to M  $$
for some finite-dimensional algebraic $P$-module $W.$ Let $\frd:=\ker(p)$ be its kernel. Then set 
$$\cF^G_P(M):=\Ind^G_P((\cO({\bf U^-_{P,\bbF}}) \otimes W)^\frd)$$
where $(\cO({\bf U^-_{P,\bbF}}) \otimes W)^\frd$ is the orthogonal complement of $\frd$ with respect to the pairing (\ref{pairing}). The latter submodule can be interpreted  as the graded dual of $M$. In particular we get 
$$\cF^G_P(M)'= \Ind^G_P(M).$$

\begin{lemma}
 Let $M$ be an object of $\dcO^\frp_\alg$. Then  $\cF^G_P(M)$ is an object of the category $Mod_{\dot{\bbF}[G,\frg]}^{d}$. Its dual $\cF^G_P(M)'$ is an object of the category $Mod_{\dot{\bbF}[G,\frg]}^{fg,d}$.
\end{lemma}

\begin{proof}
 It suffices to show the second assertion. As $G/P$ is a finite set, we need only to show that  $\cF^G_P(M)'$ has a decomposition into finite-dimensional weight spaces. Let 
 $M=\bigoplus_\lambda M_\lambda$. We write $\cF^G_P(M)=\bigoplus_{g\in G/P} \delta_g\star M$ where $\delta_{g} \star M$ is the $\dUg$-module with the same underlying vector space but where the Lie algebra action is twisted by $Ad(g)$.
  We consider the Bruhat decomposition $G/P =\bigcup_{w\in W_P} U_{B,w} w P/P$ where $U_{B,w}=U \cap wU^-w^{-1}$ and take the obvious representatives for $G/P.$ Thus we have 
  $$\cF^G_P(M)'=\bigoplus_{w \in W_P} \bigoplus_{u \in U_{B,w}^-}\delta_{uw} \star M.$$
    In the case of $\delta_w, w \in W,$ the grading
 of $\delta_w\star M$ is given by $\bigoplus_\lambda M_{w\lambda}$. In the case of $\delta_u, u\in U_{B,w}$ the grading is given by $\bigoplus_\lambda u\cdot M_\lambda$ (Note that we have an action of $U$ on $M$). In general we consider the mixture of these cases.
 \end{proof}

Let $V$ be additionally a finite-dimensional $P$-module. Then we set
$$\cF^G_P(M,V):=\Ind^G_P((\cO({\bf U^-_{P,\bbF}}) \otimes W')^\frd \otimes V).$$
This is an object of $Mod_{\dot{\bbF}[G,\frg]}^{d}$ by a slight generalization of the above lemma. In this way we get  a bi-functor
$$\cF^G_P: \dcO^\frp_\alg \times \Rep(P) \to Mod_{\dot{\bbF}[G,\frg]}^{d}.$$
By the following statement the dual $\cF^G_P(M,V)'$ is an object of $Mod_{\dot{\bbF}[G,\frg]}^{fg,d}.$

\begin{lemma}
 The dual  of $\cF^G_P(M,V)$ is given by
$$\cF^G_P(M,V)'=\dot{\bbF}[G,\frg] \otimes_{\dot{\bbF}[P,\frg]} (M \otimes V').$$
\end{lemma}
 
 \begin{proof}
We have $\cF^G_P(M,V)'=\Ind^G_P(M'\otimes V)'=
\Ind^G_P((M')'\otimes V')=\Ind^G_P(M\otimes V').$
 \end{proof}

\begin{prop}\label{exact} The functor $\cF^G_P$ is exact in both arguments.
\end{prop}

\noindent \Pf We start to prove that the functor is exact in the first argument. Let $0\rightarrow M_1\rightarrow M_2 \rightarrow M_3 \rightarrow 0$ be an exact sequence in the category $\cO_\alg^\frp.$ Then 
the sequence
$0\rightarrow \Ind^G_P M_1\rightarrow \Ind^G_P M_2 \rightarrow \Ind^G_P M_3 \rightarrow 0$
is also exact. But the graded dual of this sequence is exactly
$0\rightarrow \cF^G_P(M_3)\rightarrow \cF^G_P(M_2) \rightarrow \cF^G_P(M_1) \rightarrow 0.$

As for exactness in the second argument let 
$0\rightarrow V_1\rightarrow V_2 \rightarrow V_3 \rightarrow 0$ be an exact sequence of $P$-representations. As
$$\cF^G_P(M,V):=\Ind^G_P((\cO({\bf U^-_{P,\bbF}}) \otimes W')^\frd \otimes V_i)$$ and $\Ind^G_P$ is an exact functor we see easily the claim. \qed

Now let   ${\bf Q \supset P}$ be a parabolic subgroup and let $M \in \dcO^\frq_\alg$. Then we may consider it also as an object of 
$\dcO^\frp_\alg.$

\begin{prop}\label{PQ}
\noindent If ${\bf Q \supset P}$ is a parabolic subgroup,  $M$ an object of $\dcO^\frq_\alg$ and $V$ a finite-dimensional $P$-module, then
$$\cF^G_P(M,V) = \cF^G_Q(M,\Ind^{Q}_{P}(V)) \,.$$
\end{prop}

\vskip5pt

\noindent \Pf We have 
\begin{eqnarray*}
 \cF^G_P(M,V) & = &\Ind^G_P(M'  \otimes V)= \Ind^G_Q(\Ind^Q_P(M'  \otimes V)) \\ 
 & = & \Ind^G_Q(M'  \otimes \Ind^Q_P(V))=\cF^G_Q(M,\Ind^{Q}_{P}(V))
\end{eqnarray*}
by the projection formula. Hence we deduce the claim.
\qed

As in \cite{OS} a parabolic Lie algebra  $\frp$ is called {\it maximal} for an object $M\in \dcO^\frp$ if there does not exist a parabolic Lie algebra $\frq \supsetneq \frp$ with  $M\in \dcO^\frq.$

\begin{thm}\label{thm_irred} Let $p>3.$
 Let $M$ be an simple object of $\dcO^\frp_\alg$ such that $\frp$ is maximal   for $M$. Then
 $\cF^G_P(M)$ is  a simple  $\dot{\bbF}[G,\frg]$-module.
\end{thm}

\begin{proof}
  The proof follows the idea of loc.cit. and is even simpler.
  We start with the observation that by duality  $\cF^G_P(M,V)$  is simple as  $\dot{\bbF}[G,\frg]$-module iff $\cF^G_P(M,V)'$ is simple as  $\dot{\bbF}[G,\frg]$-module.
    We consider again the Bruhat decomposition $G/P= \bigcup_{w \in W_P} U_{B,w}^-wB/B$ and the induced decomposition
$$\cF^G_P(M)'=\bigoplus_{w \in W_P} \bigoplus_{u \in U_{B,w}^-}\delta_{uw} \star M.$$ We denote (with respect to $\delta_{uw} \star M$) for elements $\frz \in \dUg$ and $m \in M$ the action of $\frz$ on $m$ by $\frz\cdot_{uw} m.$
Now each summand $\delta_{uw} \star M$ is simple since $M$ is simple.
Thus it suffices to show that the summands are pairwise non isomorphic as $\dUg$-modules. Suppose that there is an isomorphism
$\phi: \delta_g \star M \to \delta_h \star M$ for some elements $g,h$ as above. We may suppose that $h=e.$ Write $g=u^{-1}w.$ Then such an isomorphism is equivalent to an isomorphism $\phi: \delta_w \star M \to \delta_u \star M \cong M$. The latter isomorphism is given by the mapping $m \mapsto u^{-1}\cdot m.$

 We show that this can only happen if $w \in W_P$. Let $\lambda \in X({\bf T})^\ast$ be the highest weight of $M$, i.e. $M = L(\lambda)$, and $P=P_I$ is the standard parabolic subgroup induced by  $I = \{\alpha \in \Delta \midc \langle \lambda, \alpha^\vee \rangle \in \bbZ_{\ge 0} \}$, cf. \cite{Hu}. Suppose $w$ is not contained in $W_I = W_P$. Then there is a positive root  $\beta \in \Phi^+\setminus\Phi^+_I$ such that $w^{-1}\beta < 0$, hence $w^{-1}(-\beta) > 0$. Consider a non-zero element element $y \in \frg_{-\beta}$, and let $v^+ \in M$ be a weight vector of weight $\lambda$. Then we have for $n \in \bbN,$ the following identity 
$$y^{[n]} \cdot_w v^+ = \Ad(w^{-1})(y^{[n]}) \cdot v^+ = 0$$
\noindent as ${\rm Ad}(w^{-1})(y^{[n]}) \in \frg_{-w^{-1}\beta}$ annihilates $v^+$. We have $\phi(v^+) = v$ for some nonzero $v \in M$. And therefore
$$0 = \phi(y^{[n]} \cdot_w v^+) = y^{[n]} \cdot \phi(v^+) = y^{[n]} \cdot v \;.$$
But $y$ is an element of $\fru^-_P$, hence we get a contradiction by Proposition \ref{notlocnilp} since $n$ was arbitrary.
\end{proof}

\begin{thm}\label{thm_irred_V} Let $p >3.$
 Let $M$ be an simple object of $\dcO^\frp_\alg$ such that $\frp$ is maximal for $M$ and let $V$ be an irreducible $P$-representation. Then
 $\cF^G_P(M,V)$ and its dual $\cF^G_P(M,V)'$ are simple as  $\dot{\bbF}[G,\frg]$-module.
\end{thm}

\begin{proof}

Again by duality it is enough to check the assertion for $\cF^G_P(M,V)'.$
So let $U \subset \cF^G_P(M,V)'$ be a non-zero  $G$-invariant subspace. Recall that 
$\cF^G_P(M)' =  \bigoplus_{\gamma \in G/P} \delta_\gamma \star L(\lambda)$ so that 
$$\cF^G_P(M,V) =  \bigoplus_{\gamma \in G/P} \delta_\gamma \star L(\lambda)' \otimes V^\gamma.$$
Considered as $\dUg$-module $\cF^G_P(M,V)$ is isomorphic to $(\bigoplus_{\gamma\in G/P} \delta_\gamma \star L(\lambda)') \otimes V.$ Hence by the simplicity of $M$ and since the summands $\delta_\gamma \star L(\lambda)'$ are pairwise not isomorphic 
the $\dUg$-module $U$ is equal to
$$\bigoplus_{\gamma \in G/P}
\delta_\gamma \star L(\lambda)' \otimes_\bbF V_\gamma \,,$$
with subspaces, $V_\gamma,\gamma,$ of $V$. Here $\delta_1 \star L(\lambda)' \otimes V_1=L(\lambda)' \otimes V_1$ is a $\dot{\bbF}[P,\frg]$-submodule of
$L(\lambda)' \otimes V$. Since $V$ ist irreducible the latter object is irreducible, as well. Hence $V_1=V.$ But since $G$ permutes the summands of $U$ we see that $U=\cF^G_P(M,V)'.$ 
\end{proof}

In the following statement we merely consider elements in a root space by the very definition of $\dUg.$ 
\begin{lemma}\label{lemma1} Let $p >3.$ Let $x \in \frg_\gamma$ some element for $\gamma \in \Phi$. Let $M$ be a
$\dUg$-module and $v \in M$.

\noindent (i) If $x$ acts locally finitely\footnote{Note that this definition is stronger than the one in characteristic $0$.} on $v$ (i.e., the $K$-vector space generated by $(x^{[i]}.v)_{i \ge 0}$ is finite-dimensional), then $x$ acts locally finitely on $\dUg.v$.

\noindent (ii) If $x.v = 0$ and $[x,[x,y]] = 0$ for some $y \in \frg_\beta$, where $\beta \in \Phi$ then 
$$x^{[n]} y^{[n]}.v = [x,y]^{[n]}.v \;.$$
\end{lemma}

\begin{proof} (i)
The idea is to apply Lemma 8.1 of loc.cit. which gives in characteristic 0 the formula
$$x^k \cdot z_1z_2 \ldots z_n = \sum_{
i_1 + \ldots + i_{n+1} = k} \frac{k !}{i_1 !  \ldots  i_{n+1}!} [x^{(i_1)},z_1]\cdot \ldots \cdot [x^{(i_n)},z_n]x^{i_{n+1}}.$$
Here the expression $[x^{(i)},z]$ means $ad(x)^i(z).$
We may rewrite this as
$$x^{[k]} \cdot z_1z_2 \ldots z_n = \sum_{
i_1 + \ldots + i_{n+1} = k} \frac{1}{i_1 !  \ldots  i_{n}!} [x^{(i_1)},z_1]\cdot \ldots \cdot [x^{(i_n)},z_n]x^{[i_{n+1}]}.$$
Indeed  we consider the PBW-decomposition $\dUg = \dUu \otimes \dUt \otimes \dUu$ and assume that the elements $z_i$ lie without loss of generality in one of these factors. For any element $z$ in some root space it follows from  \cite[0.2]{Hu} that $[x^{(k)},z]=0$ for all $k \geq 4.$ Since we avoid the situation $p=2,3$ we my divide my the denominators $2!$ and $3!$.

Now in contrast to loc.cit. we have again to consider $z_i$ as elements of $\dUg$ instead of elements in $\frg.$ Let $d_i$ be the order of the differential  $z_i$. Then $[x^{(i_1)},z_1]\cdots [x^{(i_n)}, z_n]$ is an differential  of order less than $4(d_1+\ldots + d_n).$ In particular we can conclude as in loc.cit. that the term lies in a finite dimensional vector space which gives now easily the claim. 

ii) In characteristic 0 we have the formula $x^{n} y^{n}.v = n! \cdot [x,y]^{n}v,$  cf. \cite[Lemma 8.2 ii)]{OS}. We only have to divide  two times by $n!.$
\end{proof}

\begin{prop}\label{notlocnilp} Let $p >3.$ Let $\frp = \frp_I$ for some $I \sub \Delta$. Suppose $M \in \dcO^\frp$ is a highest weight module with highest weight $\lambda$ and
$$I = \{\alpha \in \Delta \midc \langle \lambda, \alpha^\vee \rangle \in \bbZ_{\ge 0} \} \;.$$
\noindent Then no non-zero element of a root space of  $\fru_\frp^-$ acts locally finitely on $M$.
\end{prop}

\begin{proof}
The proof is in principal the same as  in  the case of characteristic 0 \cite[Cor. 8.2]{OS}. However we have to modify some technical ingredients of the necessary lemmas  due the different characteristic.

let $y \in (\fru_\frp^-)_\gamma$ for some root $\gamma.$
 Let $v^+$ be a weight vector with weight $\lambda$.  
 Write $\gamma = \sum_{\alpha \in \Delta} c_\alpha \alpha$ (with non-negative integers $c_\alpha$). We show by induction on the height $ht(\gamma)$ of $\gamma$ (Recall that $ht(\gamma) = \sum_{\alpha \in \Delta} c_\alpha$) that $y_\gamma$ can not act locally finite. For this it suffices by weight reasons to show that $y_\gamma^{[n]}.v^+ \neq 0$ for infinitely many  positive integers $n$.

If $ht(\gamma) = 1$, then $\gamma$ is an element of $\Delta \setminus I$. Rescaling $y_\gamma$ we can choose $x_\gamma \in \frg_\gamma$ such that $[x_\gamma,y_\gamma] = h_\gamma$ and $[h_\gamma,x_\gamma] = 2x_\gamma$ and $[h_\gamma,y_\gamma] = -2y_\gamma$. 
Then by  \cite[5.2]{Hab} we get
\begin{numequation}\label{heightone}
x_\gamma^{[n]}y_\gamma^{[n]}.v^+ = \left(\lambda(h_\gamma)  \atop n \right).v^+ = \frac{1}{n!}\prod_{i=0}^{n-1}(\langle \lambda, \gamma^\vee \rangle - i).v^+ \,.
\end{numequation}
\noindent As $I = \{\alpha \in \Delta \midc \langle \lambda, \alpha^\vee \rangle \in \bbZ_{\ge 0} \}$, it follows that $\langle \lambda, \gamma^\vee \rangle \notin \bbZ_{\ge 0}$ and the term on the right of \ref{heightone} does not vanish for infinitely many  $n.$ 
In particular, $y_\gamma^n.v^+ \neq 0$ for infinitely many $n \ge 0$.


Now suppose $ht(\gamma)>1$. Then we can write $\gamma = \alpha + \beta$ with $\alpha \in \Delta$ and $\beta \in \Phi^+$. Clearly, not both $\alpha$ and $\beta$ can be contained in $\Phi_I$. We distinguish two cases.


(a) Let  $\beta - \alpha \notin \Phi$. Then we get
for  $\alpha \notin I$ by Lemma \ref{lemma1}:
$$x_\beta^{[n]}y_\gamma^{[n]}.v^+ = [x_\beta,y_\gamma]^{[n]}.v^+ $$
where $x_\beta$ is a non-zero element of $\frg_\beta$.
We conclude by induction that $[x_\beta,y_\gamma]^{[n]}.v^+ \neq 0$ for infinitely many $n \ge 0$. 

For $\alpha \in I$  we have by Lemma \ref{lemma1}:
$$x_\alpha^{[n]}y_\gamma^{[n]}.v^+ = [x_\alpha,y_\gamma]^{[n]}.v^+ \,.$$ 
where $x_\alpha$ be a non-zero element of $\frg_\alpha$.
Again we conclude by induction the claim. And thus $y_\gamma^{[n]}.v^+ \neq 0$ for infinitely many $n \ge 0$.


(b) Let $\beta - \alpha$ is in $\Phi$. Then we have $\gamma - k\alpha \in \Phi^+$ for $0 \le k \le k_0$ (with $k_0 \le 3$, cf. \cite[0.2]{Hu}), and $\gamma - k\alpha \notin \Phi \cup \{0\}$ for $k > k_0$. This implies $[x_\alpha^{(i)},y_\gamma] = 0$ for $i>k_0$. By Lemma \ref{lemma1} we conclude as in loc.cit.
$$x_\alpha^{[nk_0]}y_\gamma^{n}.v^+ = 
\sum_{i_1 + \ldots + i_n = nk_0} \frac{1}{i_1 ! \ldots  i_n!} [x_\alpha^{(i_1)},y_\gamma]\cdot \ldots \cdot [x_\alpha^{(i_n)},y_\gamma].v^+
$$
which can be rewritten as (the corresponding term vanishes if there is one $i_j>k_0$) 
$$\frac{1}{(k_0!)^n} \; [x_\alpha^{(k_0)},y_\gamma]^{n}.v^+.$$ 
Thus we get 
$$x_\alpha^{[nk_0]}y_\gamma^{[n]}.v^+=\frac{1}{(k_0!)^n} \; [x_\alpha^{(k_0)},y_\gamma]^{[n]}.v^+.$$
If $\gamma - k_0\alpha$ is not in $\Phi_I$ we are done by induction. Otherwise we necessarily have $\alpha \notin I$. In this case, if we choose some $x_\beta \in \frg_\beta \setminus \{0\}$ and deduce as in loc.cit that 
$$x_\beta^{[n]}y_\gamma^{[n]}.v^+ = [x_\beta,y_\gamma]^{[n]}.v^+ \,,$$
As we are now in the case of height one, we can thus conclude again.
\end{proof}

\medskip

\begin{rmk}
 Unfortunately  objects in the category $\dcO$ do not have finite length in general. This holds in particular  for the local cohomology modules $H^{d-i}_{\bbP^i}(\bbP^d,\cO)$ as discussed in \cite{Ku}. However in loc.cit. it was pointed out that one can consider composition series of countable length in the sense of Birkhoff \cite{Bi}. In this way one can use similar to the $p$-adic case \cite{OS} the functors $\cF^G_P$ for a description of the composition factors of the terms  $\Ind^G_{P_{(j+1,d-j)}}(\tilde{H}^{d-j}_{\bbP^j}(\bbP^n,\cE)\otimes St_{d+1-j})$ appearing in the Theorem of Kuschkowitz.
\end{rmk}

\section{Second approach}

This section is inspired by the theory of $\cD$-modules.
Here we carry out the theory presented in the previous section 
for the rings of differential operators on the flag variety $X:={\bf B}_\bbF \backslash {\bf G}_\bbF.$ 

Let $D_{\bbP^d_\bbF}(\bbP^d_\bbF)$ be the space of global sections of the $\cD$-module sheaf $D_{\bbP^d_\bbF}$  on the projective variety $\bbP^d_\bbF$.  For a homogeneous vector bundle  
$\cE$ on $\bbP^d_\bbF$, set $$D^{\cE}_{\bbP^d_\bbF}=\cE(\bbP^d_{\bbF}) \otimes D_{\bbP^d_\bbF}(\bbP^d_{\bbF}) \otimes \cE^\ast(\bbP^d_{\bbF}).$$   Then  
$D^\cE_{\bbP^d_\bbF}$ acts naturally on $\cE(\cX)$  and  the filtration appearing in Kuschkowitz's theorem. Instead we consider (which become clear later) the space of global sections $D=D_{X}(X)$ of the differential operators on $X$ and 
$$D^\cE=\cE(X) \otimes D \otimes \cE(X)$$ for any homogeneous vector bundle $\cE$ on $B\backslash G.$
There is an action of $D^\cE$ on all the above objects as well. 
We consider further the Beilinson-Bernstein homomorphism 
$$\pi^\cE:\dUg \rightarrow D^\cE$$
which is not surjective (for $\cE=\cO_X$) in positive characteristic as shown by Smith in \cite{Sm}.

Consider the covering
$X=\bigcup_{w\in W} B\backslash B U^-w$ by translates of the big open cell $B\backslash B U^-$. Let $D^1=D(B\backslash B U^-).$ Thus $D^1$ is the crystalline Weyl algebra
$$D^1=\bbF[T_\alpha \mid \alpha \in \Phi^- ]\langle y^{[n]}_\alpha \mid \alpha \in \Phi^- , n\in \bbN \rangle.$$
By the sheaf property we see that  
$D$ coincides with the set
\begin{equation}\label{sheaf_projective_covering}
 \{\Theta \in D^1 \mid \Theta(\cO(B\backslash B U^-w)) \subset \cO(B\backslash B U^-w) \;\forall w  \}.
\end{equation}

For any prime power $q=p^n$ we let $D^1_q$ be the differential operators which are $\bbF[T_\alpha^q \mid \alpha \in \Phi^-]$-linear. Then we have $D=\bigcup_n D_{p^n}.$
The next statement is a generalization of \cite[lemma 3.1]{Sm}.
We set for $\alpha >0$, $T_\alpha:=T_{-\alpha}^{-1}.$ 

\begin{lemma}\label{lemma_diff}
 Let $\Theta \in D^1_q$. Then $\Theta \in D$ iff
 
 i) $\Theta(1) \in \bbF$
 
 and 
 
 ii) $\Theta(\prod_{\alpha\in \Phi^-} T_\alpha^{i_\alpha})\in V:=\bigoplus_{0\leq j_\alpha \leq q} \prod_{\alpha\in \Phi^-} T_\alpha^{j_\alpha}$ for all tuples 
 $(i_\alpha)_\alpha$ with $0 \leq i_\alpha \leq q-1.$ 
\end{lemma}

\begin{proof}
 $\Rightarrow:$ The first item follows from the sheaf property (\ref{sheaf_projective_covering}) since $\cO(B\backslash G)=\bbF$. Now let $\Theta \in D \cap D^1_q.$   Let $w_0\in W$ be the longest element and 
 $f=\prod_{\alpha<0} T_\alpha^{i_\alpha}$ as above. 
  Then $g=f \cdot\prod_{\alpha >0} T_\alpha^q  \in \cO(B\backslash BU^-{w_0})$. But then 
  $$\Theta(f)=(\prod_{\alpha<0 } T_\alpha^q)\Theta(g)\in   ( \prod_{\alpha} T_{\alpha<0}^q) \cO(B\backslash BU^-{w_0}) \cap  \cO(B\backslash BU^-)\subset V.$$

 $\Leftarrow:$ We show that  $\Theta(\cO(B\backslash BU^- w)) \subset \cO(B\backslash BU^-w)$ $\forall w\in W$. We consider the element $f=\prod_{\beta \in w(\Phi^-)} T^{i_\beta}_\beta \in \cO(B\backslash BU^-{w}).$  Write  
 $$f=\prod_{\beta \in w(\Phi^-) \atop \beta <0} T^{i_\beta}_\beta  \prod_{\beta \in w(\Phi^-) \atop \beta >0} T^{i_\beta}_\beta= \prod_{\beta \in w(\Phi^-) \atop \beta <0} T^{i_\beta}_\beta  \prod_{\beta \in w(\Phi^-) \atop \beta >0} T^{-i_\beta}_{-\beta}.$$ 
 For each $\beta>0$ let $m_\beta$ be the integer with $m_\beta q < i_\beta  \leq (m_\beta +1)q.$
 On the other hand,  for each $\beta<0$ let $m_\beta$ be the integer with $m_\beta q \leq i_\beta  < (m_\beta +1)q.$
  Then $\prod_{\beta\in w(\Phi^-) \atop \beta <0 } T_\beta^{i_\beta} =    
 \prod_{\beta\in w(\Phi^-) \atop \beta <0 } T_\beta^{m_\beta q} T_\beta^{i_\beta- m_\beta q}.$
    Putting this together we get by assumption (ii)
    $$\Theta(\prod_{\beta\in w(\Phi^-) \atop \beta >0 } T_{-\beta}^{(m_\beta+1)q-i_\beta} \prod_{\beta\in w(\Phi^-) \atop \beta <0 } T_\beta^{i_\beta- m_\beta q}) \in V.$$
  Thus $\Theta (f) \in  \prod_{\beta\in w(\Phi^-) \atop \beta >0 } T_{-\beta}^{-(m_\beta+1)q}  \prod_{\beta\in w(\Phi^-) \atop \beta <0 } T_\beta^{m_\beta q} V \subset \cO(B\backslash BU^-{w}).$ 
 \end{proof}

We fix the same setup as in the previous section. I.e.
 ${\bf P\subset  G}$ is a parabolic subgroup,  ${\bf U_P}$ its unipotent radical and  ${\bf U^-_P}$ its opposite unipotent radical. Moreover we have fixed as before  lifts  ${\bf P_{\bbZ}}$ etc. inside ${\bf G_{\bbZ}}$. We consider
 the following subalgebras of $D$ in terms of generators:

$D(P)=\langle T_\alpha^m\cdot  y_{\alpha}^{[n]}\in D  \mid m \leq n \mbox{ for }  y_{\alpha} \in \frp \cap \frb^-, m \geq n \mbox{ for }  L_{- \alpha} \in \fru \rangle.$

$D(U_P)=\langle (T_\alpha)^m\cdot  y_{\alpha}^{[n]}\in D  \mid m > n, L_{-\alpha} \in \fru_P \rangle.$

$D(U_P^-)=\langle (T_\alpha)^m\cdot  y_{\alpha}^{[n]}\in D  \mid m < n, y_{\alpha} \in \fru_P^- \rangle.$

$D(L_P)=\langle (T_\alpha)^m\cdot  y_{\alpha}^{[n]}\in D  \mid m \leq n \mbox{ for }  y_{\alpha} \in \frl_P \cap \frb^-, m > n \mbox{ for }  L_{- \alpha} \in \frl_P \cap \fru \rangle.$

$D(T)=\langle (T_\alpha)^m\cdot  y_{\alpha}^{[n]}\in D  \mid m = n,  \alpha \in \Delta \rangle.$

\medskip

\begin{rmk} 
i)  Note that  $D(T)$ is  for $p\neq 2$ nothing else but $\pi^{\cO_X}(\dUt)$ as $T_\alpha y_\alpha=\pi(2h_\alpha)$ for all $\alpha \in \Delta.$ Hence if $\lambda \in X^\ast(T)$, it induces a $D(T)$-module structure  on $\bbF$ which we denote by $\bbF_\lambda.$


ii) By Lemma \ref{lemma_diff} one checks that $D(U_P)=\pi^{\cO_X}(\dUup)$ since $T_\alpha^2 y_\alpha =\pi(L_{-\alpha}) \forall \alpha\in \Phi^-.$
 \end{rmk}

\begin{lemma} There is for all $n\in \bbN$ and $\alpha \in \Delta$ the identity
$ \left( T_\alpha y_\alpha \atop  n \right)=T^n_\alpha y_\alpha^{[n]}.$
\end{lemma}

\begin{proof}
 This is left to the reader.
\end{proof}

We set $D^\cE(P)=\cE(X) \otimes D(P) \otimes \cE^\ast(X)$ etc. Then there is a product decomposition $D^\cE=D^\cE(P)D^\cE(U_P^-)$ (an almost PBW-decomposition).

Again we mimic the definition of the category $\cO$ in the sense of BGG. Let $\cO^{P}_{D^\cE}$ be the category of  $D^\cE$-modules such that

i) $M$ is finitely generated as a $D^\cE$-module 

ii) As a $D^\cE(L_P)$-module it is a direct sum of finite-dimensional modules.

iii) $D^\cE(U_P)$  acts locally finite-dimensional, i.e. for all $m\in M$ the subspace $D^\cE(U_P)\cdot v$ is finite-dimensional.

\medskip

\begin{rmk}
 For $\cE=\cO_X$ this category corresponds in analogy to the classical case to the principal block.
\end{rmk}

We define the algebraic part of $\cO^P_{D^\cE,\alg}$ as usual, i.e. we denote by $\cO_{D^\cE,\alg}^P$ the full subcategory of $\cO_{D^\cE}^P$ consisting of objects such that the action of $\dUt$ on the weight spaces is given by algebraic characters $\lambda \in X^\ast(T).$  Note that axioms ii) and iii) induce together with
the map $\pi^\cE: \dUg \to D^\cE$ an algebraic  $P$-module structure on any object in $\cO^P_{D^\cE,\alg}$.


As in the classical case we see that the axioms imply the existence of a finite-dimensional $D^\cE(P)$-module $N$ which generates $M$ as a $D^\cE$-module.  Further there are similar definitions. E.g. a  vector in an $D^\cE$-module $M\in \cO_{D^\cE}$ is called a maximal vector of weight $\lambda \in \frt^\ast$  if $v \in M_\lambda$ and  $D^\cE (U_P)\cdot  v = 0.$
 A $D^\cE$-module $M$ is called  a highest weight module of
weight $\lambda$ if there is a maximal vector $v \in  M_\lambda$  such that 
$M = D^\cE \cdot v.$
By the very definition such a module satisfies 
$M = D^\cE(U^-_B) \cdot  v.$ 
For a   one-dimensional $\dUt$-module $\lambda$  we consider it as usual via the trivial $D^\cE(U_B)$-action as a one-dimensional $D^\cE(B)$-module $\bbF_\lambda$ and set
$M(\lambda)=D^\cE \otimes_{D^\cE(B)} \bbF_\lambda. $
More generally we may define for every finite-dimensional $D^\cE(P)$-module $W$ the generalized Verma module
$M(W)=D^\cE \otimes_{D(P)} W.$
Note that we have surjections
$ D^\cE(U^-_{B}) \otimes {\bar \bbF}_\lambda \to M(\lambda)$
and
$D^\cE(U^-_{P}) \otimes_\bbF W \to M(W).$
We see by the above surjections that \cite[Thm. 1.3]{Hu} holds true in our category, i.e. if $M(\lambda)\neq 0$ then it has a unique simple quotient $L(\lambda)$.  Moreover these modules form a complete list  of simple modules in the "union" of  our categories $\cO_{D^\cE}.$

Consider the local cohomology module $\tilde{H}^{d-j}_{\bbP^j}(\bbP^d,\cO).$
For $d-j\geq 2$ this coincides with the vector space of polynomials  
$$\bigoplus_{{n_0,\ldots,n_j \geq 0 \atop n_{j+1} \ldots  n_d < 0} \atop \sum_i n_i=0} \bbF \cdot X_0^{n_0}\cdots X_j^{n_j} X_{j+1}^{n_{j+1}} \cdots X_d^{n_d}$$ cf. \cite{O2}. 
In general there is some finite-dimensional ${\bf P_{(j+1,d-j)}}$-module $V$ such that $\tilde{H}^{d-j}_{\bbP^j}(\bbP^d,\cE)$ is a quotient of $\bigoplus_{{n_0,\ldots,n_j \geq 0 \atop n_{j+1} \ldots  n_d \leq 0} \atop \sum_i n_i=0} \bbF \cdot X_0^{n_0}\cdots X_j^{n_j} X_{j+1}^{n_{j+1}} \cdots X_d^{n_d} \otimes V.$

\begin{prop}
Let $\cE$ be a homogeneous vector bundle on $\bbP^d_\bbF$.  Then  $\tilde{H}^{d-j}_{\bbP^j}(\bbP^d,\cE)$ is an object of $\cO^{P_{(j+1,d-j)}}_{D^\cE} $.
\end{prop}

\begin{proof}
The non-trivial aspect is to show that $\tilde{H}^{d-j}_{\bbP^j}(\bbP^d,\cE)$ is finitely generated. We will show this for $\cE=\cO$.  We claim that 
 $$\bigoplus_{n_0,\ldots,n_j \geq 0 \atop \sum_{i=0}^j n_i=d-j} \bbF \cdot X_0^{n_0}\cdots X_j^{n_j} X_{j+1}^{-1} \cdots X_d^{-1}$$
 is as in characteristic $0$  a generating system of $H^{d-j}_{\bbP^j}(\bbP^d,\cO).$ Indeed, as in the latter case we can apply successively the differential operators $L_\alpha\in \fru_{P_{(j+1,d-j)}}^-$ to obtain all expressions $X_0^{n_0}\cdots X_j^{n_j} X_{j+1}^{n_{j+1}} \cdots X_d^{n_d}$ such that $|n_i|\leq p$ for all $i \geq j+1.$ 
  In order to obtain those where $n_i=-(p+1)$ for some $i\geq j+1$  we can apply  $y^{[p]}_{(-,j+1)}$ to get the desired denominators. However, we do not get all  possible nominators. 
   But in our algebra $D$ we have in contrast to $\dUg$  the differential operator $T_{(a,b)}^{p-1} L_{(a,b)}^{[p]}$ with $j+1 \leq a <b \leq d$ at our disposal. Applying these operators we    can realize all nominators. 
 For $|n_i|>p+1$ in particular for $|n_i|=rp+1, r \geq 2$ we use the same method as above etc..
\end{proof}

\begin{prop}
 The object $\tilde{H}^i_{\bbP^j}(\bbP^d,\cO)$ is a simple module  isomorphic to  $L(s_i\cdots s_1 \cdot 0).$
  \end{prop}

\begin{proof}
In characteristic $0$ we  gave a proof in \cite[Prop. 7.5]{OS}.
Here we can argue with the differential operators at our disposal in the same way. Note that for general $\lambda \in X^\ast(T)$ the simple module $L(\lambda)$ is an avatar of the characteristic $0$ version.
 \end{proof}

We let 
$$\cA^\cE_G:=\bbF[G]\# D^\cE$$ be the smash product of the group algebra $\bbF[G]$ and $D^\cE.$

Let $M$ be an object of $\cO_{D^\cE, \alg}^P$ and let $V$ be a finite-dimensional $P$-module. Then we set
$$\cF^G_P(M,V):=\bbF[G] \otimes_{\bbF[P]} (M \otimes V).$$ 
Note that $\cF^G_P(M,V)=\Ind^G_P(M \otimes V)$.
This is a $\cA^\cE_G$-module. In this way we get a bi-functor
$$\cF^G_P:\cO_{D^\cE,\alg}^P \times \Rep(P) \to Mod_{\cA^\cE_G}.$$

The proof of the next statement is the same as in Propositions \ref{exact} and \ref{PQ}.

\medskip
\begin{prop} a) The bi-functor $\cF^G_P$ is exact in both arguments.

\noindent b) If $Q \supset P$ is a parabolic subgroup, 
$M$ an object of $\cO_{D^\cE,\alg}^Q$, then
$$\cF^G_P(M,V) = \cF^G_Q(M,\Ind^{Q}_{P}(V)) \,,$$
\noindent where $\Ind^Q_P(V)$  denotes the corresponding induced representation. \qed
\end{prop}

\vskip5pt

\begin{thm}
 Let $M$ be an simple object of $\cO^P_{D^\cE,\alg}$ such that $P$ is maximal for $M$ and let $V$ be a simple $P$-representation. Then
 $\cF^G_P(M,V)$ is simple as  $\cA^\cE_G$-module.
\end{thm}

\begin{proof}
The proof follows the strategy of Theorems \ref{thm_irred} and \ref{thm_irred_V}. Note that Proposition \ref{notlocnilp}  does also hols true for our objects $L(\lambda)$ as avatars of their character zero versions. 
\end{proof}

\end{document}